\documentclass[11pt,reqno]{amsart}
\usepackage{amsmath,amssymb,bbm,url}

\usepackage{tikz}
\usepackage[bookmarks=false,unicode,colorlinks,urlcolor=red,citecolor=red,linkcolor=blue,pagebackref]{hyperref}
\usepackage{aliascnt}

\newtheorem*{problem}{Problem}

\numberwithin{equation}{section}

\newtheorem{theorem}{Theorem}[section]

\newaliascnt{lemma}{theorem}
\newtheorem{lemma}[lemma]{Lemma}
\aliascntresetthe{lemma}

\newaliascnt{proposition}{theorem}
\newtheorem{proposition}[proposition]{Proposition}
\aliascntresetthe{proposition}

\newaliascnt{corollary}{theorem}
\newtheorem{corollary}[corollary]{Corollary}
\aliascntresetthe{corollary}

\newaliascnt{conjecture}{theorem}
\newtheorem{conjecture}[conjecture]{Conjecture}
\aliascntresetthe{conjecture}

\newcommand{\e}{\varepsilon}

\newcommand{\R}{{\mathbb R}}
\newcommand{\Rd}{{\R^d}}

\newcommand{\tr}{\operatorname{tr}}

\newcommand{\Z}{{\mathbb Z}}
\newcommand{\Zd}{{\Z^d}}

\newcommand{\mref}[1]{%
\href{http://www.ams.org/mathscinet-getitem?mr=#1}{#1}}

\newcommand{\arxiv}[1]{%
\href{http://front.math.ucdavis.edu/#1}{ArXiv:#1}}

\begin{document}

\title[]{Sums of Laplace eigenvalues --- rotations and tight frames in higher dimensions}

\author[]{R. S. Laugesen and B. A. Siudeja}
\address{Department of Mathematics, University of Illinois, Urbana,
IL 61801, U.S.A.} \email{Laugesen\@@illinois.edu}
\email{Siudeja\@@illinois.edu}
\date{\today}


\begin{abstract}
The sum of the first $n \geq 1$ eigenvalues of the Laplacian is shown to be maximal among simplexes for the regular simplex (the regular tetrahedron, in three dimensions), maximal among parallelepipeds for the hypercube, and maximal among ellipsoids for the ball, provided the volume and moment of inertia of an ``inverse'' body are suitably normalized. This result holds for Dirichlet, Robin and Neumann eigenvalues. Additionally, the cubical torus is shown to be maximal among flat tori.

The method of proof involves tight frames for euclidean space generated by the orbits of the rotation group of the extremal domain.

The ball is conjectured to maximize sums of Neumann eigenvalues among general convex domains, provided the volume and moment of inertia of the polar dual of the domain are suitably normalized.
\end{abstract}

\maketitle

\vspace*{-12pt}

\section{\bf Introduction}

In this paper we show how the \emph{Method of Rotations and Tight Frames} developed in our earlier paper for plane domains \cite{LS11b} can be extended to yield higher dimensional bounds on sums of eigenvalues of the Laplacian. The transition to higher dimensions requires a broader algebraic context, in terms of irreducible symmetry groups. Furthermore, the results must be reformulated in a non-obvious way, because the obvious generalization fails in every dimension higher than two.

We will obtain sharp bounds on the sum of the first $n\ge 1$ eigenvalues of linear images of highly symmetric domains. This work generalizes P\'olya's results \cite{P52,PS53} in three ways. His proof could handle only the first Dirichlet eigenvalue of planar domains; our method bounds eigenvalue sums of arbitrary length, in arbitrary dimensions, and under any major boundary condition (Dirichlet, Neumann or Robin). For a detailed comparison of P\'olya's approach with our own, see the end of Section~6 in \cite{LS11b}.

Write $\lambda_1$, $\lambda_2$, \dots for the Dirichlet eigenvalues of the Laplacian on a domain in $\Rd$. Let $V$ be the volume and $I$ the second moment of mass around the centroid (which in two dimensions simply equals the moment of inertia about a perpendicular axis through the centroid). We show that for any linear transformation $T$ and any domain $D$ possessing sufficient rotational symmetry, the normalized eigenvalue sum
\begin{equation} \label{introeq}
  \left.(\lambda_1+\cdots+\lambda_n)V^{2/d}\right|_{T(D)} \left.\frac{V^{1+2/d}}{I}\right|_{T^{-1}(D)}
\end{equation}
is maximal when $T$ is the identity matrix. That is, the original domain $D$ is the maximizer.

Both factors in expression \eqref{introeq} are scale- and rotation-invariant. Hence maximality holds also when $T$ is any multiple of an orthogonal matrix.

Interestingly, the moment of mass normalization in \eqref{introeq} is imposed on the ``inverse domain'' $T^{-1}(D)$, rather than on the domain $T(D)$ where the eigenvalues are computed. This feature seems unavoidable in dimensions three and higher, as the rectangular box example reveals in \autoref{precresults}. In two dimensions one can normalize the moment of inertia on $T(D)$ \cite{LS11b}.

We obtain similar results for Neumann and Robin boundary conditions, and for flat tori. The Robin case requires a significant extension of the proof from the planar situation.

\smallskip
There are at least two good reasons for studying eigenvalue sums. The first is physical: the sum of eigenvalues represents the energy needed to fill the lowest $n$ quantum states under the Pauli exclusion principle. The second is mathematical: summability methods provide a tools for studying high eigenvalues, which are difficult to study individually. An example of the gains possible by  summation is that while P\'olya's conjecture (claiming the Weyl asymptotic is a lower bound for each Dirichlet eigenvalue) is still open, Li and Yau \cite{LY82} have shown that \emph{sums} of eigenvalues are indeed bounded below by the analogous Weyl asymptotic.

Notice our work in this paper is geometrically sharp, with an extremal domain existing for each fixed $n$. The Li--Yau inequality is not geometrically sharp, but is asymptotically sharp --- equality holds for each domain in the limit as $n \to \infty$. Asymptotically sharp upper bounds in the Neumann case were obtained by Kr\"oger \cite{Kr94}, and sums of functions of eigenvalues were studied recently by  Frank, Geisinger, Harrell, Hermi, Laptev, Loss and Weidl (\cite{DLL08,FLLS06,GLW10,HH08} and references therein).

Eigenvalue estimates for rotationally symmetric domains were previously studied by Ashbaugh and Benguria \cite{AB93}, while some higher dimensional upper bounds for the first Dirichlet eigenvalue were proved by Freitas and Krej{\v{c}}i{\v{r}}{\'{\i}}k \cite{FK08}. Their result applies to general domains, however it is usually weaker for the domains we can handle. Nevertheless, it is equivalent for the first Dirichlet eigenvalue of ellipsoids. Eigenvalues of certain simplexes have been found explicitly \cite{Ja08,KMV82,Tu84}, although the regular tetrahedron is not one of them because the trigonometric method of Lam\'{e} for the equilateral triangle fails to extend to higher dimensions.

For a detailed literature survey regarding eigenvalue estimates, one could begin with our two-dimensional paper \cite{LS11b} and Henrot's book \cite{He06}.

\section{\bf Assumptions and notation}\label{secassump}
\label{notation}

\subsection*{Eigenvalues}
For a bounded domain $D$ in $\Rd, d \geq 2$, we denote the Dirichlet eigenvalues of the Laplacian by $\lambda_j(D)$, the Robin eigenvalues by $\rho_j(D)$, and the Neumann eigenvalues by $\mu_j(D)$. In the Robin and Neumann cases, we assume the domain has Lipschitz boundary so that the spectrum is discrete. Denoting the eigenfunctions by $u_j$ in each case, we have
\[
\begin{cases}
-\Delta u_j = \lambda_j u_j \;\; \text{in $D$} \\
\hfill u_j = 0 \;\; \text{on $\partial D$}
\end{cases}
\;\;
\begin{cases}
-\Delta u_j = \rho_j u_j \;\; \text{in $D$} \\
\hfill \frac{\partial u_j}{\partial n}+\sigma u_j  = 0 \;\; \text{on $\partial D$}
\end{cases}
\;\;
\begin{cases}
-\Delta u_j = \mu_j u_j \;\; \text{in $D$} \\
\hfill \frac{\partial u_j}{\partial n}  = 0 \;\; \text{on $\partial D$}
\end{cases}
\]
and
\[
0 < \lambda_1 < \lambda_2 \leq \lambda_3 \leq \dots \quad
0 < \rho_1 < \rho_2 \leq \rho_3 \leq \dots \quad\;\;
0 = \mu_1 < \mu_2 \leq \mu_3 \leq \dots
\]

\subsection*{Geometric quantities}
Let
\begin{itemize}
  \item[] $V=$ volume,
  \item[] $I=$ second moment of mass about the centroid.
\end{itemize}
That is,
\[
  I(D) = \int_D |x-\overline{x}|^2 \, dx
\]
where the centroid is $\overline{x} = \int_D x \, dx/V$.

Note that $I$ can be interpreted as an average moment of inertia with respect to a randomly chosen axis (see for example \cite{FLL07}). It can also be viewed as the moment of inertia of $D$ in $\R^{d+1}$, with respect to the axis perpendicular to the hyperplane containing $D$ and passing through the centroid.

Given a matrix $M$, write its Hilbert--Schmidt norm as
\[
\lVert M \rVert_{HS}=\big( \sum_{j,k} M_{jk}^2 \big)^{\! 1/2} = (\tr M^\dagger M)^{1/2} ,
\]
where $M^\dagger$ denotes the transposed matrix.

\subsection*{Tight frames}
The leading role in our proofs is played by the concept of a tight frame:

\smallskip
A set of unit vectors $\{y_i\}_{i=1}^N$ in $\Rd$ forms a unit-norm \emph{tight frame} if for every vector $z \in \Rd$,
\begin{equation} \label{tightframeeq}
    \frac{1}{N}\sum_{i=1}^N |z \cdot y_i|^2 =\frac{1}{d} |z|^2.
\end{equation}

\smallskip
This property resembles the Plancherel identity for an orthonormal basis, except that the number of vectors $N$ can exceed the dimension $d$.

Tight frames have become an important tool in applied harmonic analysis, in recent years. The overdetermined representations they provide are useful for noise reduction and robustness to erasures, in signal processing. Note that the tight frames in this paper consist of equal-norm vectors; for more on that special case, see the work of Benedetto and Fickus \cite{BF03} and Casazza and Kova\v{c}evi\'{c} \cite{CK03}. General tight frames do not impose the equal-norm restriction. All tight frames arise as (rescaled) projections of orthonormal bases in higher dimensional spaces \cite[Chapter~5]{H07}. For this and more information about frame theory in finite and infinite dimensional spaces, one may consult the monographs of Christensen \cite{Ch03} and Han \emph{et al.}\ \cite{H07}.

\section{\bf Sharp upper bounds on eigenvalue sums}\label{precresults}

A \emph{symmetry of $D$} is an orthogonal matrix $U$ that maps $D$ to itself (with the matrix acting on the left by $x \mapsto Ux$). A group $\mathcal{U}$ of orthogonal matrices is \emph{irreducible} if every nontrivial orbit spans $\Rd$, that is, if $\{ Ux : U \in \mathcal{U} \}$ spans $\Rd$ for every $x \in \Rd \setminus \{ 0 \}$. Equivalently, the group is irreducible if the only subspaces of $\Rd$ that are invariant under the action of the group are $\Rd$ and the zero subspace.

\subsection*{Dirichlet and Neumann boundary conditions}
The next theorem shows how a linear transformation affects the eigenvalues of a domain having irreducible symmetry group.
\begin{theorem} \label{higherdimDN}
If the symmetry group of $D$ is irreducible, then
\[
(\lambda_1 + \dots + \lambda_n) \big|_{T(D)} \leq
\frac{1}{d} \lVert T^{-1} \rVert_{HS}^2 (\lambda_1 + \dots + \lambda_n) \big|_D
\]
for each $n \geq 1$ and each invertible linear transformation $T$ of $\Rd$. Equality holds if $T$ is a scalar multiple of an orthogonal matrix.

The same inequality holds for the Neumann eigenvalues.
\end{theorem}
The theorem and its corollaries below are proved in \autoref{higherdimproofs}. The two dimensional version of the result was proved already in our previous paper  \cite[Theorem~3.1]{LS11b}. In two dimensions, the irreducibility hypothesis simply means $D$ has $N$-fold rotational symmetry for some $N \geq 3$. Further, in two dimensions we found necessary and sufficient conditions for equality for the fundamental tone ($n=1$). Here in \autoref{higherdimDN} we state a sufficient condition for equality, but we have not succeeded in finding necessary conditions.

Stretching a domain by the same factor in all directions causes the eigenvalues to scale monotonically. The next corollary obtains a monotonicity estimate when the domain is stretched by different factors in different directions. Such estimates are especially interesting in the Neumann case, where domain monotonicity is unavailable.
\begin{corollary}[Stretching and monotonicity] \label{higherstretching}
Suppose the symmetry group of $D$ is irreducible. Let $T$ be a diagonal matrix with entries $t_1,\ldots,t_d>0$, so that $T(D)$ stretches $D$ by a factor $t_i$ in the $i$th coordinate direction. Then
\[
( \lambda_1 + \dots + \lambda_n) \big( T(D) \big) \leq \frac{t_1^{-2}+\cdots+t_d^{-2}}{d} \, ( \lambda_1 + \dots + \lambda_n)(D), \qquad n \geq 1 .
\]
(Equality holds if $t_1 = \dots = t_d$.) This result holds also for the Neumann eigenvalues.

In particular, if $t_i > 1$ for each $i$, then the eigenvalue sums of the stretched domain are smaller than the eigenvalue sums of the original domain.
\end{corollary}
The next corollary expresses our theorem in more geometric terms. We state it only for linear images of certain regular solids, for simplicity, although analogous results hold for any domain $D$ having an irreducible group of symmetries, such as Archimedean solids and demihypercubes.
\begin{corollary}[Simplexes, parallelepipeds, and ellipsoids] \label{higherregularDN}
Let $D$ be a regular simplex (respectively: hypercube, ball) in $\Rd$. Among all simplexes (respectively: parallelepipeds, ellipsoids) of the form $T(D)$, where $T$ is an invertible linear transformation, the quantity
\[
\left. (\lambda_1 + \dots + \lambda_n) V^{2/d} \right|_{T(D)} \left. \frac{V^{1+2/d}}{I} \right|_{T^{-\dagger}(D)}
\]
is maximal for the regular simplex (respectively: hypercube, ball), for each $n \geq 1$.

The same result holds for Neumann eigenvalues.
\end{corollary}
Other maximizers are possible too, in some cases. In particular, each rectangular box centered at the origin is a maximizer when $n=1$ and $D$ is a hypercube, as we will show in an example below.

\subsection*{Understanding the geometric factors}
An unusual feature of the corollary is that the eigenvalues are computed on one domain, $T(D)$, while the scale-invariant geometric factor $V^{1+2/d}/I$ is computed on a different domain, $T^{-\dagger}(D)$. The Example below explains why it is natural to evaluate the eigenvalues and the geometric factor on different domains.

The geometric factor could be evaluated on $T^{-1}(D)$ (without the transpose), because the transpose changes neither the volume nor the second moment of mass of the domain (see \autoref{higherdimHSnorm} and note that $T^{-1}$ and its transpose have the same Hilbert--Schmidt norm). We have chosen to state the ``transposed'' version of the corollary because it leads to a conjecture for polar duals, in \autoref{open}.

In two dimensions, \autoref{higherregularDN} reduces to saying that $(\lambda_1 + \cdots + \lambda_n)A^3/I$ is maximal for the equilateral triangle among triangles (respectively: square among parallelograms, disk among ellipses), as one sees by taking $d=2$, $V=A$, and using an identity connecting the moments of inertia of $T(D)$ and $T^{-1}(D)$ in two dimensions \cite[Lemma~5.4]{LS11b}. This two dimensional version of \autoref{higherregularDN} appeared in our earlier paper \cite[Corollary~3.2]{LS11b}.

An obvious generalization of this two dimensional estimate would be to maximize $(\lambda_1 + \cdots + \lambda_n)V^{1+4/d}/I$, but this functional is unbounded even on the family of rectangular boxes in three dimensions, as the next example shows.

\subsubsection*{Example: Rectangular boxes}
Suppose $D$ is the unit cube centered at the origin in three dimensions, and let $T$ be a diagonal matrix with positive entries $t_1,t_2,t_3$; see \autoref{fig5}. Then $T(D)$ is a rectangular box with side lengths $t_1,t_2,t_3$, volume $t_1t_2t_3$ and first Dirichlet eigenvalue
\[
\lambda_1 = \pi^2(t_1^{-2}+t_2^{-2}+t_3^{-2}) .
\]
And $T^{-\dagger}(D)$ is a box with side lengths $t_1^{-1},t_2^{-1},t_3^{-1}$, volume $(t_1t_2t_3)^{-1}$ and second moment of mass
\[
(t_1^{-2}+t_2^{-2}+t_3^{-2})/12t_1t_2t_3 .
\]
    \newcommand{\pudlo}[7]{
    \draw (#1,0) rectangle +(#2,#3);
    \draw (#1,#3) -- ++(#4,#4) -- ++(#2,0) -- ++(0,-#3) -- ++(-#4,-#4);
    \draw (#1+#2+#4,#3+#4) -- (#1+#2,#3);
    \draw[dashed] (#1,0) -- ++(#4,#4) -- ++(0,#3);
    \draw[dashed] (#1+#4,#4) -- ++(#2,0);
    \draw (#1,#3/2) node [left=-2pt] {\tiny $#6$};
    \draw (#1+#2/2,0) node [below] {\tiny $#5$};
    \draw (#1+#2+#4/2,#4/2) node [below right=-3pt] {\tiny $#7$};
    }
    \begin{figure}[h]
      \begin{center}
  \begin{tikzpicture}[scale=0.9]
    \pudlo{5}{1}{1}{1/3}{1}{1}{1}
    \draw[<-] (2.5,1) .. controls (3.5,1.3) .. (4.5,1) node [above,pos=0.5] {\tiny $T$};
    \pudlo{0}{1.5}{1}{1/4.5}{t_1}{t_2}{t_3}
      \draw (1,1.5) node { $T(D)$};
      \draw (5.7,1.6) node {$D$};
    \begin{scope}[draw=red]
      \pudlo{10}{1/1.5}{1}{1/2}{t_1^{-1}}{t_2^{-1}}{t_3^{-1}}
      \draw[->] (7,1) .. controls (8,1.3) .. (9,1) node [above,pos=0.5,red] {\tiny $T^{-\dagger}$};
      \draw (10.5,1.8) node { $T^{-\dagger}(D)$};
    \end{scope}
    \draw (-1,-0.3) node [below right] {};
    \draw (12.5,-0.5) node [below left] {\small $I=(t_1^{-2}+t_2^{-2}+t_3^{-2} )V/12$};
    \draw (-1,-0.5) node [below right] {\small $I=(t_1^2+t_2^2+t_3^2)V/12$};
  \end{tikzpicture}
      \end{center}
      \caption{The first eigenvalue $\lambda_1=\pi^2 (t_1^{-2}+t_2^{-2}+t_3^{-2})$ of $T(D)$ is compatible with the second moment of mass of $T^{-\dagger}(D)$, not the moment of $T(D)$.}
      \label{fig5}
    \end{figure}
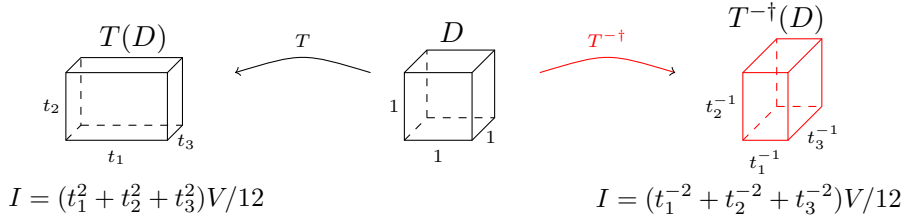
The sum of reciprocal side length squares appears in both these last two displayed formulas. These factors cancel to give
\[
\left. \lambda_1 V^{2/3} \right|_{T(D)} \cdot \left. \frac{V^{1+2/3}}{I} \right|_{T^{-\dagger}(D)}
= 12 \pi^2 ,
\]
regardless of the side lengths $t_1,t_2,t_3$. Thus:
\begin{quote}
every rectangular box maximizes the first Dirichlet eigenvalue ($n=1$) in \autoref{higherregularDN}.
\end{quote}

Now choose $t_1=t_2=1$ and $t_3=\e$, so that the box $T(D)$ is a cube that has been squashed in one direction. We calculate
\[
\left. \lambda_1 \frac{V^{1+4/3}}{I} \right|_{T(D)} = 12 \pi^2 \frac{(2+\e^{-2}) \e^{4/3}}{2+\e^2} ,
\]
which tends to $\infty$ as $\e \to 0$. Hence normalizing the first eigenvalue by the volume and moment solely of $T(D)$ fails to yield a bounded quantity, even on the class of rectangular boxes.

The preceding example helps explain why the second moment of mass should be evaluated on the ``inverse'' of the domain on which the eigenvalues are evaluated. Or to put it heuristically, the eigenvalue $\lambda_1$ is dominated by contributions from the ``short'' directions of $T(D)$. These short directions contribute little to the moment of mass of $T(D)$ but contribute a lot to the moment of mass of $T^{-\dagger}(D)$, since they correspond to ``long'' directions in that domain .

\subsection*{Robin boundary condition}
\begin{theorem}\label{thmrobin}
If the symmetry group of $D$ is irreducible and $T$ is an invertible linear transformation of $\Rd$, then
\[
(\rho_1 + \dots + \rho_n) \big|_{\sigma\lVert T^{-1} \rVert_{HS}/\sqrt{d}, \, T(D)} \leq \frac{1}{d} \lVert T^{-1} \rVert_{HS}^2 (\rho_1 + \dots + \rho_n) \big|_{\sigma,D}
\]
for each $n \geq 1$. Equality holds if $T$ is a scalar multiple of an orthogonal matrix.
\end{theorem}
The subscript $\sigma,D$ indicates both the Robin constant $\sigma$ and the domain $D$. Incidentally, the Robin parameter on $T(D)$ in the theorem is multiplied by $\lVert T^{-1} \rVert_{HS}$ to ensure that the eigenvalues scale correctly with respect to $T$.

A simpler and more ``geometric'' estimate on the Robin eigenvalues can be deduced by fixing the volume of the domain, as in the next corollary.
\begin{corollary}\label{corrobin}
If the symmetry group of $D$ is irreducible and $T$ is an invertible linear transformation with $|\det T|=1$, then for each $n \geq 1$,
\[
\left. (\rho_1 + \dots + \rho_n) V^{2/d} \right|_{\sigma,\, T(D)} \left. \frac{V^{1+2/d}}{I} \right|_{T^{-\dagger}(D)}
\]
is maximal when $T$ is an orthogonal matrix.
\end{corollary}
Notice the Robin parameter in this corollary is independent of the transformation $T$.

\subsection*{Flat torus}
Our final result concerns the eigenvalues of a flat torus. Suppose $T$ is an invertible linear transformation on $\Rd$, so that $T\Zd$ is a lattice and $\Rd/T\Zd$ is a flat torus.  Denote the eigenvalues of the Laplacian on this torus by $\tau_j$, so that
\[
0 = \tau_1 < \tau_2 \leq \tau_3 \leq \dots
\]
and the eigenfunctions satisfy $-\Delta u_j = \tau_j u_j$ on the torus. Some extremal results on those eigenvalues were obtained by Nadirashvili \emph{et al.}\ \cite{JLNNP05,N96}. These eigenvalues are given explicitly as $4\pi^2$ times the squares of lengths of vectors in the dual lattice  $T^{-\dagger} \Zd$, because each vector $y$ in the dual lattice generates an eigenfunction $e^{2\pi i y \cdot x}$.
\begin{proposition} \label{flattori}
Among all flat tori $\Rd/T\Zd$, where $T$ is an invertible linear transformation, for each $n \geq 2$ the normalized eigenvalue sum
\[
(\tau_2 + \dots + \tau_n) \frac{1}{\lVert T^{-\dagger} \rVert_{HS}^2}
\]
is maximal when $T$ is a scalar multiple of an orthogonal matrix (that is, when the torus is cubical).
\end{proposition}
The Hilbert--Schmidt norm in this corollary measures the ``size'' of the fundamental parallelepiped $T^{-\dagger} \big( [0,1]^d \big)$ of the dual lattice $T^{-\dagger} \Zd$, since the norm is computed from the sum of the squares of the lengths of the column vectors of $T^{-\dagger}$. Thus the corollary can be interpreted as maximizing a ratio of sums of squares of lengths of vectors in the dual lattice: the numerator is the sum of squares of the lengths of the $n$ shortest vectors in that lattice, while the denominator is the sum of squares of the lengths of $d$ edges of the fundamental parallelepiped. To obtain the best possible bound, for a given lattice, we should choose $T$ so that the chosen fundamental parallelepiped has minimal sum of squares of spanning vector lengths.

\section{\bf Open problem --- general convex domains} \label{open}
We want to extend our eigenvalue bounds to general convex domains, generalizing from the class of linear images of rotationally symmetric domains.

The \emph{polar dual} of a bounded convex domain $\Omega$ in $\Rd$ containing the origin is the bounded convex domain
\[
\Omega^\circ = \{ x \in \Rd : x \cdot y < 1 \text{\ for all\ } y \in \overline{\Omega} \} .
\]
The polar dual of a ball of radius $r$ centered at the origin is a ball of radius $1/r$. The polar dual of a square is a square rotated by $\pi/4$. The dual of a regular polygon is a rotated regular polygon. However, in three dimensions the polar dual of a cube centered at the origin is an octahedron. Hence the polar dual need not preserve the shape of a regular solid.

If $T$ is an invertible linear transformation then the domain $T(\Omega)$ has polar dual $T^{-\dagger}(\Omega^\circ)$. For more on polar duals, see \cite[Sec.~2.8]{W94}.

%
%
%
%

These examples suggest that taking a \emph{polar dual} is similar to passing from $T(D)$ to $T^{-\dagger}(D)$. Therefore it seems reasonable to try maximizing the quantity
\begin{equation} \label{convexdual}
\left. (\lambda_1 + \dots + \lambda_n) V^{2/d} \right|_\Omega \left. \frac{V^{1+2/d}}{I} \right|_{\Omega^\circ} .
\end{equation}
This quantity is certainly bounded above, as follows. An ellipsoid $E$ can be constructed inside $\Omega$ such that a rescaled ellipsoid $d \, E$ contains $\Omega$ (John's ellipsoid \cite[Chapter 4.5 Section 2]{GW93}). Then $(dE)^\circ \subset \Omega^\circ \subset E^\circ$. Therefore quantity \eqref{convexdual} can be estimated above and below, using domain monotonicity, by a constant times the same quantity for $E$. And \eqref{convexdual} for the ellipsoid $E$ is bounded above by the value for the ball, by \autoref{higherregularDN}. Thus the expression \eqref{convexdual} is bounded above on the class of convex domains containing the origin.
\begin{problem}
Find convex domains containing the origin that maximize \eqref{convexdual}, especially for the fundamental tone ($n=1$).
\end{problem}

For Neumann eigenvalues, boundedness of the counterpart of \eqref{convexdual} follows from the relation $\mu_j \leq \lambda_j$ between the Neumann and Dirichlet eigenvalues. As to a maximizer, we suggest:
\begin{conjecture} \label{higherDNconj}
Suppose $\Omega \subset \Rd$ is a bounded convex domain containing the origin. Then the scale-invariant Neumann eigenvalue sum
\[
\left. (\mu_2 + \dots + \mu_n) V^{2/d} \right|_\Omega \left. \frac{V^{1+2/d}}{I} \right|_{\Omega^\circ}
\]
is maximal when $\Omega$ is a ball, for each $n \geq 2$.
\end{conjecture}
In two dimensions, the Conjecture claims that the product
\begin{equation} \label{neusum1}
(\mu_2 + \dots + \mu_n) A \Big|_\Omega \cdot \left. \frac{A^2}{I} \right|_{\Omega^\circ}
\end{equation}
is maximal for the disk. When $\Omega$ is a centered ellipse, parallelogram, or triangle, this last product equals the simpler-looking quantity
\begin{equation} \label{neusum2}
(\mu_2+\cdots+\mu_n)\frac{A^3}{I} \Big|_\Omega
\end{equation}
by \autoref{momentratio} below. Thus for these special domains in two dimensions, the Conjecture is equivalent to Conjecture~4.2 in \cite{LS11b}. For general convex domains, the quantities \eqref{neusum1} and \eqref{neusum2} are at least comparable above and below (since the ratios $A^2/I$ on $\Omega$ and $\Omega^\circ$ are comparable, using John's ellipsoid once more), but the relation between the two conjectures is unclear to us.

\section{\bf Proofs} \label{higherdimproofs}

The Rayleigh quotients for each boundary condition are defined as follows
\begin{align*}
\text{Dirichlet:} \qquad R[u] &  = \frac{\int_D |\nabla u|^2 \, dx}{\int_D u^2 \, dx} && \text{for\ } u \in H^1_0(D) , \\
\text{Robin:} \qquad R[u] &  = \frac{\int_D |\nabla u|^2 \, dx+ \sigma \int_{\partial D} u^2\, ds}{\int_D u^2 \, dx} && \text{for\ } u \in H^1(D) , \\
\text{Neumann:} \qquad R[u] &  = \frac{\int_D |\nabla u|^2 \, dx}{\int_D u^2 \, dx} && \text{for\ } u \in H^1(D) .
\end{align*}
The \textbf{Rayleigh--Poincar\'{e} Principle} characterizes the sum of the first $n \geq 1$ Dirichlet eigenvalues as:
\begin{align*}
& \lambda_1 + \dots + \lambda_n \\
& = \min \big\{ R[v_1] + \dots + R[v_n] : v_1, \dots,v_n \in H^1_0(D)\text{\ are pairwise orthogonal in $L^2(D)$} \big\} ,
\end{align*}
and similarly for sums of Neumann or Robin eigenvalues, using trial functions in $H^1$ instead of $H^1_0$. (See \cite[p.~98]{B80}.)

\subsection*{Tight frame identities} Our proofs will use a Plancherel-type ``tight frame'' identity. Before stating the identity, we note that the symmetry group of a domain in euclidean space is compact, as one sees by a short argument (using that the symmetry group of the domain equals the symmetry group of the complement of the domain, which is a closed set).

Let $\mathcal{U}$ be the symmetry group of a domain $D$, so that $\mathcal{U}$ is compact. Write $\nu$ for the Haar probability measure on $\mathcal{U}$. If the group is finite then Haar measure is simply the normalized counting measure.
\begin{lemma} \label{higherdimtightframe}
Let $\mathcal{U}$ be the symmetry group of a domain $D$. Then $\mathcal{U}$ is irreducible if and only if for all row vectors $z \in \Rd$ and all matrices $Y$ with $d$ rows one has
\begin{equation} \label{VWeq}
  \int_{\mathcal{U}} |zUY|^2 \, d\nu(U) = \frac{1}{d} |z|^2 \lVert Y\rVert_{HS}^2.
\end{equation}
In case $\mathcal{U}$ is finite, the last formula says
\begin{equation} \label{VWeq1}
  \frac{1}{|\mathcal{U}|} \sum_{U \in \mathcal{U}} |zUY|^2 = \frac{1}{d} |z|^2 \lVert Y\rVert_{HS}^2.
\end{equation}
\end{lemma}
If the matrix $Y$ has only one column and has unit norm, then equation \eqref{VWeq1} reduces to the tight frame property \eqref{tightframeeq} for the system of unit vectors $\{ UY : U \in \mathcal{U} \}$. This case was proved by Vale and Waldron \cite{VW04} (see also \cite{VW05,VW08}). As they remarked, the result is essentially Schur's Lemma, and we give such a proof below.
\begin{proof}
``$\Longrightarrow$''   Note that
\begin{align}
     \int_{\mathcal{U}} |zUY|^2 \, d\nu(U) & = \int_{\mathcal{U}} (zUY)(zUY)^{\dagger} \, d\nu(U) \notag \\
     & =z \left( \int_{\mathcal{U}} UYY^{\dagger}U^{\dagger} \, d\nu(U) \right)z^{\dagger}=zMz^{\dagger} , \label{tightmatrix}
\end{align}
say, where the matrix $M$ is clearly symmetric. Furthermore for any $U \in \mathcal{U}$, we see $MU=UM$ by the group property of $\mathcal{U}$. Let $\alpha$ be an eigenvalue of $M$ with eigenvector $w$. Then
\[
  M(Uw)=U(Mw)=U(\alpha w)=\alpha(Uw).
\]
Hence the entire orbit $\{ Uw : U \in \mathcal{U} \}$ consists of eigenvectors belonging to $\alpha$. The orbit spans all of $\Rd$, by the irreducibility hypothesis, and so $M=\alpha \, \mathrm{Id}$.
Taking the trace yields
\[
\alpha d = \tr M=\int_{\mathcal{U}} \tr(UYY^{\dagger}U^{\dagger}) \, d\nu(U) =\int_{\mathcal{U}} \tr(YY^{\dagger}) \, d\nu(U) = \lVert Y\rVert_{HS}^2.
\]
Solving for $\alpha$ and substituting $M=\alpha \, \mathrm{Id}$ into \eqref{tightmatrix} proves \eqref{VWeq}.

``$\Longleftarrow$'' If $\mathcal{U}$ is reducible then for some vector $y \neq 0$ (that is, for some nonzero matrix $Y$ having $d$ rows and $1$ column) the orbit $\{ Uy : U \in \mathcal{U} \}$ does not span $\Rd$, and so some nonzero vector $z$ is orthogonal to every element of the orbit. The left side of \eqref{VWeq} then equals zero while the right side does not, meaning \eqref{VWeq} fails for this $z$ and $Y$.
\end{proof}

\subsection*{Proof of \autoref{higherdimDN}}
The proof goes like in two dimensions \cite[Theorem~3.1]{LS11b}, except that the tight frame identity is more sophisticated. We provide a complete proof, for the reader's convenience.

We will prove the Dirichlet case. The idea is to obtain trial functions on the domain $T(D)$ by linearly transplanting eigenfunctions of $D$, and then to average with respect to the rotations of $D$. The Neumann proof is identical, except for using Neumann eigenfunctions.

Let $u_1,u_2,u_3,\ldots$ be orthonormal eigenfunctions on $D$ corresponding to the Dirichlet eigenvalues $\lambda_1,\lambda_2,\lambda_3,\ldots$. Consider a symmetry (orthogonal matrix) $U$ that fixes $D$. Define trial functions
\[
v_j = u_j \circ U \circ T^{-1}
\]
on the domain $E=T(D)$.

The functions $v_j$ are pairwise orthogonal, since
\[
\int_E v_j v_k \, dx = \int_D u_j u_k \, dx \cdot |\det TU^{-1}| = 0
\]
if $j \neq k$. Thus by the Rayleigh--Poincar\'{e} principle, we have
\[
\sum_{j=1}^n \lambda_j(E) \leq \sum_{j=1}^n \frac{\int_E |\nabla v_j|^2 \, dx}{\int_E v_j^2 \, dx} .
\]
For each function $v=v_j$ we evaluate the last Rayleigh quotient as
\begin{align*}
\frac{\int_E |\nabla v|^2 \, dx}{\int_E v^2 \, dx}
& = \frac{\int_D |(\nabla u)(x)UT^{-1}|^2 \, dx \cdot |\det TU^{-1}|}{\int_D u^2 \, dx \cdot |\det TU^{-1}|} \\
& = \int_D |(\nabla u) UT^{-1}|^2 \, dx ,
\end{align*}
where the gradient $\nabla u$ is regarded as a row vector and in the last line we used that $u=u_j$ is normalized in $L^2(D)$.

By averaging the preceding equality over all symmetries $U$ in the group ${\mathcal U}$ of symmetries of $D$, we find
\begin{align*}
\sum_{j=1}^n \lambda_j(E)  & \leq \int_{\mathcal{U}} \sum_{j=1}^n \int_D  |(\nabla u_j) U T^{-1}|^2 \, dx \, d\nu(U) \\
& = \sum_{j=1}^n \int_D \Big\{ \frac{1}{d} |\nabla u_j|^2 \lVert T^{-1} \rVert_{HS}^2 \Big\} \, dx \qquad \text{by \autoref{higherdimtightframe}} \\
& = \frac{1}{d} \lVert T^{-1} \rVert_{HS}^2 \sum_{j=1}^n \lambda_j(D) ,
\end{align*}
which proves the inequality in \autoref{higherdimDN}.

Obviously equality holds in the theorem if $T$ is orthogonal, since the eigenvalues of the Laplacian are invariant under orthogonal transformations and the Hilbert--Schmidt norm of an orthogonal matrix equals $\sqrt{d}$. Equality also holds for scalar multiples of orthogonal transformations, by a simple scaling argument.

\subsection*{Geometric interpretation of the Hilbert--Schmidt norm}
To prove \autoref{higherregularDN} we must express the Hilbert--Schmidt norm of $T^{-1}$ in terms of volume and the second moment of mass.
\begin{lemma} \label{higherdimHSnorm}
If $D$ has an irreducible symmetry group and $T$ is an invertible $d \times d$ matrix, then
\[
\frac{1}{d} \lVert T^{-1} \rVert_{HS}^2 = \left.  \frac{V(D)^{1+4/d}}{I(D)} \right/  \frac{V \big( T(D) \big)^{2/d} \, V \big( T^{-\dagger}(D) \big)^{1+2/d}}{I \big( T^{-\dagger}(D) \big)} .
\]
\end{lemma}
In two dimensions we were able to develop a simpler formula \cite{LS11b} solely in terms of the domain $T(D)$, due to a relation between the Hilbert--Schmidt norms of $T$ and its inverse ($\lVert T \rVert_{HS} = |\det T| \lVert T^{-1} \rVert_{HS}$). No such relation seems to hold in higher dimensions.
\begin{proof}[Proof of \autoref{higherdimHSnorm}]
The centroid of $D$ lies at the origin, as a consequence of the irreducibility of the symmetry group. Hence the centroid of $T^{-\dagger}(D)$ also lies at the origin, by a linear change of variable.

The moment matrix of $D$ is defined to be $M(D) = \int_D x x^\dagger \, dx$, where $x$ is a column vector. We claim $M(D)$ is a scalar multiple of the identity. For let $U$ be a symmetry of $D$.  The invariance of $D$ under $U$ implies that $M(D)=UM(D)U^\dagger$, so that $M$ is a scalar multiple of the identity by arguing like in the proof of \autoref{higherdimtightframe} (that is, Schur's Lemma again).

Since the diagonal entries of $M(D)$ are equal, they must be $1/d$ times the trace of the moment matrix, that is, $1/d$ times the second moment of mass (using here that the centroid of $D$ lies at the origin). Thus
\begin{equation} \label{momentidentity}
M(D) = \frac{1}{d} I(D) \, \text{Id.}
\end{equation}

The moment of inertia of $T^{-\dagger}(D)$ can now be computed as
\begin{align*}
I(T^{-\dagger} D)
& = \tr M(T^{-\dagger} D) \\
& = \tr \big( T^{-\dagger} M(D) T^{-1} |\det T^{-1}| \big) \\
& =\frac{1}{d} I(D) \, \big( \tr T^{-\dagger} T^{-1} \big) |\det T^{-1}| \qquad \text{by \eqref{momentidentity}} \\
& = \frac{1}{d} I(D) \, \lVert T^{-1} \rVert_{HS}^2 \, |\det T^{-1}| .
\end{align*}
Rearranging,
\begin{equation} \label{HSdagger}
\frac{1}{d} \lVert T^{-1} \rVert_{HS}^2 = \frac{1}{|\det T^{-1}|} \frac{I(T^{-\dagger} D)}{I(D)} .
\end{equation}

The lemma now follows easily, using the formula
\[
|\det T^{-1}| = \frac{V(TD)^{2/d} \, V(T^{-\dagger} D)^{1+2/d}}{V(D)^{1+4/d}} .
\]
\end{proof}

Next we divert from our main argument to observe that certain plane domains deviate from roundness to the same extent as their polar duals, when the deviation is measured by the moment of inertia.
\begin{lemma} \label{momentratio}
If $\Omega$ is an ellipse, parallelogram, or triangle, and the centroid of $\Omega$ lies at the origin, then
\[
\left. \frac{A^2}{I} \right|_{\Omega^\circ} = \left. \frac{A^2}{I} \right|_{\Omega} .
\]
\end{lemma}
\begin{proof}[Proof of \autoref{momentratio}]
Write $D$ for a disk, square, or equilateral triangle, with centroid at the origin, in dimension $d=2$. Then by the examples in \autoref{open}, we know the polar dual $D^\circ$ has the same shape as $D$, up to rotation and scaling. Hence the scale invariant ratio $A^2/I$ takes the same value for $D^\circ$ as for $D$.

The domain $\Omega$ is a linear image of $D$, meaning $\Omega=T(D)$ for some invertible linear transformation $T$. We compute
\begin{align*}
\frac{A^2}{I} \big( T(D) \big)
& = \frac{2|\det T|}{\lVert T \rVert_{HS}^2} \frac{A(D)^2}{I(D)} \qquad \text{by \eqref{HSdagger} with $T$ replaced by $T^{-\dagger}$} \\
& = \frac{2|\det T|^{-1}}{\lVert T^{-1} \rVert_{HS}^2} \frac{A(D)^2}{I(D)} \\
& \qquad \qquad \text{because $\lVert T^{-1} \rVert_{HS}^2 = \lVert T \rVert_{HS}^2 / |\det T|^2$ in two dimensions} \\
& = \frac{2|\det T^{-1}|}{\lVert T^{-1} \rVert_{HS}^2} \frac{A(D^\circ)^2}{I(D^\circ)} \qquad \text{since $D$ and $D^\circ$ have the same shape,} \\
& = \frac{A^2}{I} \big( T^{-\dagger}(D^\circ) \big)
\end{align*}
by \eqref{HSdagger}, which proves the lemma because $T^{-\dagger}(D^\circ)=T(D)^\circ=\Omega^\circ$.
\end{proof}

\subsection*{Proof of \autoref{higherregularDN}}
For the ball, irreducibility of the symmetry group is obvious since every rotation is a symmetry of the ball. If $D$ is a regular simplex or a hypercube, then the symmetry group is known to be irreducible; for example, see \cite[Chapter 6]{Dav08}, where irreducibility is shown for the symmetry groups of the Platonic solids and higher dimensional regular polytopes.

\autoref{higherregularDN} therefore follows from \autoref{higherdimDN} and \autoref{higherdimHSnorm}.

\subsection*{Proof of \autoref{thmrobin}}
The Robin case in higher dimensions is significantly more involved than in the planar case, because to handle the boundary term in the Robin Rayleigh quotient we must study the Jacobian of $T$ on hyperplanes tangent to the boundary of $D$. (In two dimensions the boundary has only one tangent direction, and the boundary Jacobian is easily understood.)

We may assume the linear transformation $T$ is diagonal with positive entries $t_1,\dots,t_d$, since the general case may be reduced to this situation by the singular value decomposition of $T$ and the invariance of Robin eigenvalues under rotations of the domain.

We will study the action of $T$ on a hyperplane by examining its action on the normal vector, motivated by the fact that the base area of a parallelepiped can be determined from the volume and the altitude. Consider vectors $w_1,w_2, \dots ,w_{d-1}$ and $w$. Define $W=[w_1 \dots w_{d-1} ]$ to be the matrix with $i$th column $w_i$, and let
\[
S[W,w] = \det [w_1 \, \dots \, w_{d-1} \, w] .
\]
Write $\mathbf{e}$ for the column vector whose $i$th component is $e_i$, the $i$th  standard unit column vector. Define
\[
S[W]=S[W,\mathbf{e}] = \det [w_1 \, \dots \, w_{d-1} \, \mathbf{e}] ,
\]
so that $S[W]$ is a vector perpendicular to each vector $w_i$. Notice $|S[W] \cdot w| = |S[W,w]|$ gives the volume of the parallelepiped spanned by vectors $w_1, \dots, w_{d-1},w$. By choosing $w$ to be a unit vector orthogonal to each $w_i$, we deduce that $|S[W]|$ equals the $(d-1)$-dimensional volume of the $(d-1)$-dimensional parallelepiped spanned by $w_1,\dots,w_{d-1}$.

One can check from the definition that
\begin{equation} \label{deteq1}
  S[TW] = S[W,T^{-1}\mathbf{e}] (\det T) .
\end{equation}
Further, for any orthogonal matrix $U$ one has
\begin{equation} \label{deteq2}
  US[W] = S[UW]
\end{equation}
by the geometric interpretation of $S[W]$ as a vector perpendicular to each column of $W$ and with magnitude equalling the $(d-1)$-volume of the parallelepiped spanned by $w_1,\dots,w_{d-1}$.

Now we can prove the theorem. The proof goes like for the Dirichlet and Neumann cases in the proof of \autoref{higherdimDN}, except that for the Robin eigenvalues we must take account also of a boundary integral in the Rayleigh quotient. Following the notation of that proof, the boundary term is
\begin{align*}
\frac{\int_{\partial E} v^2 \, dS(x)}{\int_E v^2 \, dx}
& = \frac{\int_{\partial E} u(UT^{-1}x)^2 \, dS(x)}{\int_E u(UT^{-1}x)^2 \, dx} \\
& = \frac{\int_{\partial D} u(x)^2 \big| S[TU^{-1}W(x)] \big| \, dS(x)}{\int_D u(x)^2 \, dx \cdot |\det T|}
\end{align*}
by a change of variable, where $W(x)=[w_1 \dots w_{d-1}]$ is a matrix whose columns form an orthonormal basis for the tangent space of $D$ at $x$. Hence
\begin{equation} \label{rayleighrobin}
\frac{\int_{\partial E} v^2 \, dS(x)}{\int_E v^2 \, dx}
= \int_{\partial D} u(x)^2 \big| S[U^{-1}W(x),T^{-1}\mathbf{e}] \big| \, dS(x) ,
\end{equation}
by calling on \eqref{deteq1} and using the normalization of $u$ in $L^2(D)$.

The Jacobian factor in \eqref{rayleighrobin} satisfies
\begin{align*}
\big| S[U^{-1}W,T^{-1}\mathbf{e}] \big|^2
& = \sum_{i=1}^d \left( \frac{S[U^{-1}W]_i}{t_i} \right)^{\! 2} \\
& = \sum_{i=1}^d \frac{\big(e_i \cdot U^{-1}S[W]\big)^2}{t_i^2}
\end{align*}
by \eqref{deteq2}. Integrating over all symmetries $U$ and then applying the tight frame property from \autoref{higherdimtightframe} shows that
\begin{align*}
\int_{\mathcal{U}} \big| S[U^{-1}W,T^{-1}\mathbf{e}] \big|^2 \, d\nu(U)
&= \sum_{i=1}^d \frac{1}{t_i^2} \int_{\mathcal{U}} \big(e_i \cdot U^{-1}S[W]\big)^2 \, d\nu(U)\\
& = \sum_{i=1}^d \frac{1}{t_i^2} \, \frac{1}{d} |e_i|^2 \big| S[W] \big|^2 \\
& = \frac{\lVert T^{-1} \rVert_{HS}^2}{d} ,
\end{align*}
since orthonormality of the vectors $w_1,\dots,w_{d-1}$ guarantees that $|S[W]|=1$. Hence by taking the square root and using Cauchy--Schwarz,
\[
\int_{\mathcal{U}} \big| S[U^{-1}W,T^{-1}\mathbf{e}] \big| \, d\nu(U) \leq \frac{\lVert T^{-1} \rVert_{HS}}{\sqrt{d}} .
\]
Applying this last estimate to \eqref{rayleighrobin} gives that
\[
\int_{\mathcal{U}} \frac{\int_{\partial E} v^2 \, dS(x)}{\int_E v^2 \, dx} \, d\nu(U) \leq \frac{\lVert T^{-1} \rVert_{HS}}{\sqrt{d}} \int_{\partial D} u(x)^2  \, dS(x) .
\]

The adaptation of the proof of \autoref{higherdimDN} to the Robin situation can now be completed without difficulty.

\subsection*{Proof of \autoref{corrobin}}
The inequality between quadratic and geometric means implies that
\begin{align*}
\frac{\lVert T^{-1} \rVert_{HS}}{\sqrt{d}}&=\sqrt{\frac{1}{d} \sum_{i=1}^d (\text{singular value of $T^{-1}$})_i^2 } \\
& \geq \sqrt[d]{\prod_{i=1}^d (\text{singular value of $T^{-1}$})_i } = \sqrt[d]{|\det T^{-1}|} = 1,
\end{align*}
by our hypothesis that $|\det T|=1$. By inserting this last inequality into \autoref{thmrobin} and invoking the monotonicity of the Robin Rayleigh quotient (and hence eigenvalues) with respect to the Robin parameter, we find
\[
(\rho_1 + \dots + \rho_n) \big|_{\sigma,T(D)} \leq
\frac{1}{d} \lVert T^{-1} \rVert_{HS}^2 (\rho_1 + \dots + \rho_n) \big|_{\sigma,D} .
\]
Now \autoref{higherdimHSnorm} completes the proof.

\subsection*{Proof of \autoref{flattori}}
Write $v_1,\dots,v_n$ for a collection of $n$ shortest vectors in the cubical lattice $\Zd$ (so that $v_1=0$). Then since $T^{-\dagger}v_1,\dots,T^{-\dagger}v_n$ are distinct vectors belonging to the dual lattice $T^{-\dagger} \Zd$ of $T\Zd$, we have
\[
(\tau_1 + \dots + \tau_n) \big|_{\Rd/T\Zd}
\leq 4\pi^2 \left( |T^{-\dagger}v_1|^2 + \dots + |T^{-\dagger}v_n|^2 \right) .
\]
Let $U$ be any symmetry of the hypercube $[-1/2,1/2]^d$, so that $U$ maps $\Zd$ to itself. Then another collection of $n$ shortest vectors in $\Zd$ is $Uv_1,\dots,Uv_n$. Repeating the above estimate with this new collection, we find
\[
(\tau_1 + \dots + \tau_n) \big|_{\Rd/T\Zd}
\leq 4\pi^2 \left( |T^{-\dagger}Uv_1|^2 + \dots + |T^{-\dagger}Uv_n|^2 \right) .
\]
Averaging this inequality over all symmetries $U$ of the hypercube implies (by the tight frame identity in \autoref{higherdimtightframe}) that
\begin{align*}
(\tau_1 + \dots + \tau_n) \big|_{\Rd/T\Zd}
& \leq 4\pi^2 \left( |v_1|^2 + \dots + |v_n|^2 \right) \frac{\lVert T^{-\dagger} \rVert_{HS}^2}{d} \\
& = (\tau_1 + \dots + \tau_n) \big|_{\Rd/\Zd} \, \frac{\lVert T^{-1} \rVert_{HS}^2}{\lVert \text{Id}^{-1} \rVert_{HS}^2} .
\end{align*}
%



\begin{thebibliography}{99}
\bibitem{AB93}M.~S. Ashbaugh and R.~D. Benguria, \emph{Universal bounds for the low  eigenvalues of {N}eumann {L}aplacians in {$n$} dimensions},  SIAM J. Math.  Anal. 24 (1993), no.~3, 557--570. \mref{MR1215424}

\bibitem{B80}C.~Bandle, \emph{Isoperimetric Inequalities and Applications}, Monographs and  Studies in Mathematics, vol.~7, Pitman (Advanced Publishing Program), Boston,  Mass., 1980. \mref{MR572958}

\bibitem{BF03} J. J. Benedetto and M. Fickus, \emph{Finite normalized tight frames}, Adv. Comput. Math. 18 (2003), no. 2-4, 357–-385. \mref{MR1968126}

\bibitem{CK03}
P. G. Casazza and J. Kova\v{c}evi\'{c}, \emph{Equal-norm tight frames with erasures}, Adv. Comput. Math. 18 (2003), no. 2-4, 387–430. \mref{MR1968127}

\bibitem{Ch03}O.~Christensen, \emph{An Introduction to Frames and Riesz Bases}, Applied and  Numerical Harmonic Analysis, Birkh\"auser Boston Inc., Boston, MA, 2003.  \mref{MR1946982}

\bibitem{Dav08}M.~W. Davis, \emph{The Geometry and Topology of Coxeter Groups}, London  Mathematical Society Monographs Series, vol.~32, Princeton University Press,  Princeton, NJ, 2008. \mref{MR2360474}

\bibitem{DLL08}J.~Dolbeault, A.~Laptev, and M.~Loss, \emph{Lieb--Thirring inequalities with improved constants}, J. Eur. Math. Soc. (JEMS) 10 (2008), no.~4, 1121--1126. \mref{MR2443931}

\bibitem{FLLS06}R.~L. Frank, A.~Laptev, E.~H. Lieb, and R.~Seiringer, \emph{Lieb-{T}hirring inequalities for Schr\"{o}dinger operators with complex-valued potentials}, Lett. Math. Phys. 77 (2006), no.~3, 309--316. \mref{MR2260376}

\bibitem{FK08}P.~Freitas and D.~Krej{\v{c}}i{\v{r}}{\'{\i}}k, \emph{A sharp upper bound for  the first Dirichlet eigenvalue and the growth of the isoperimetric constant  of convex domains},  Proc. Amer. Math. Soc. 136 (2008), no.~8,  2997--3006. \mref{MR2399068}

\bibitem{FLL07}P.~Freitas, R.~S. Laugesen, and G.~F. Liddell, \emph{On convex surfaces with  minimal moment of inertia},  J. Math. Phys. 48 (2007), no.~12,  122902, 21. \mref{MR2377828}

\bibitem{GLW10}L.~Geisinger, A.~Laptev, and T.~Weidl, \emph{Geometrical versions of improved Berezin--Li--Yau iequalities}. \arxiv{1010.2683}

\bibitem{GW93} P.~M. Gruber and J.~M. Wills, eds. \emph{Handbook of convex geometry. Vol. A, B}, North-Holland Publishing  Co., Amsterdam, 1993. \mref{MR1242973}

\bibitem{H07} D. Han, K. Kornelson, D. Larson and E. Weber,
\emph{Frames for Undergraduates}, Student Mathematical Library, 40. American Mathematical Society, Providence, RI, 2007. \mref{MR2367342}

\bibitem{HH08}E.~M. Harrell and L.~Hermi, \emph{Differential inequalities for Riesz  means and Weyl-type bounds for eigenvalues},  J. Funct. Anal. 254  (2008), no.~12, 3173--3191. \mref{MR2418623}

\bibitem{He06}A.~Henrot, \emph{Extremum Problems for Eigenvalues of Elliptic Operators},  Frontiers in Mathematics, Birkh\"auser Verlag, Basel, 2006. \mref{MR2251558}

\bibitem{Ja08}S.~R. Jain, \emph{Exact solution of the Schr\"{o}dinger equation for a particle in a regular $N$-simplex}, Phys. Lett. A 372 (2008), no.~12, 1978--1981. \mref{MR2398154}

\bibitem{JLNNP05}D.~Jakobson, M.~Levitin, N.~Nadirashvili, N.~Nigam, and I.~Polterovich,  \emph{How large can the first eigenvalue be on a surface of genus two?},  Int. Math. Res. Not. (2005), no.~63, 3967--3985. \mref{MR2202582}

\bibitem{KMV82}H.~R. Krishnamurthy, H.~S. Mani, and H.~C. Verma, \emph{Exact solution of the Schr\"{o}dinger equation for a particle in a tetrahedral box}, J. Phys. A 15 (1982), no.~7, 2131--2137. \mref{MR665828}

\bibitem{Kr94}P.~Kr{\"o}ger, \emph{Estimates for sums of eigenvalues of the Laplacian}, J. Funct. Anal. 126 (1994), no.~1, 217--227. \mref{MR1305068}


\bibitem{LS11b}R.~S. Laugesen and B.~A. Siudeja, \emph{{Sums of Laplace eigenvalues --- rotationally symmetric maximizers in the plane}}. J. Funct. Anal., to appear. \arxiv{1009.5326}

\bibitem{LY82}P.~Li and S.~T. Yau, \emph{A new conformal invariant and its applications to  the Willmore conjecture and the first eigenvalue of compact surfaces},  Invent. Math. 69 (1982), no.~2, 269--291. \mref{MR674407}

\bibitem{N96}N.~Nadirashvili, \emph{Berger's isoperimetric problem and minimal immersions of  surfaces},  Geom. Funct. Anal. 6 (1996), no.~5, 877--897.  \mref{MR1415764}

\bibitem{P52}G.~P\'{o}lya, \emph{Sur le r\^{o}le des domaines sym\'{e}triques dans le calcul de  certaines grandeurs physiques},  C. R. Acad. Sci. Paris 235 (1952),  1079--1081. \mref{MR0052477}

\bibitem{PS53}G.~P\'{o}lya and M.~Schiffer, \emph{Convexity of functionals by  transplantation},  J. Analyse Math. 3 (1954), 245--346, With an  appendix by Heinz Helfenstein. \mref{MR0066530}

\bibitem{Tu84}J.~W. Turner, \emph{On the quantum particle in a polyhedral box}, J. Phys. A 17 (1984), no.~14, 2791--2797. \mref{MR771766}

\bibitem{VW04}R.~Vale and S.~Waldron, \emph{The vertices of the platonic solids are tight  frames}, Advances in constructive approximation: {V}anderbilt 2003, Mod.  Methods Math., Nashboro Press, Brentwood, TN, 2004, pp.~495--498.  \mref{MR2089946}

\bibitem{VW05}R.~Vale and S.~Waldron, \emph{Tight frames and their symmetries},  Constr.  Approx. 21 (2005), no.~1, 83--112. \mref{MR2105392}

\bibitem{VW08}R.~Vale and S.~Waldron, \emph{Tight frames generated by finite nonabelian  groups},  Numer. Algorithms 48 (2008), no.~1-3, 11--27.  \mref{MR2413275}

\bibitem{W94}R.~Webster, \emph{Convexity}, Oxford Science Publications, The Clarendon Press Oxford University Press, New York, 1994. \mref{MR1443208}

\end{thebibliography}
\end{document}